\documentclass[conference]{IEEEtran}
\IEEEoverridecommandlockouts
\usepackage{cite}
\usepackage{amsmath,amssymb,amsfonts,amsthm}
\DeclareMathOperator*{\argmin}{arg\,min}
\usepackage[ruled,vlined,linesnumbered]{algorithm2e}
\usepackage{graphicx}
\usepackage{textcomp}
\def\BibTeX{{\rm B\kern-.05em{\sc i\kern-.025em b}\kern-.08em
    T\kern-.1667em\lower.7ex\hbox{E}\kern-.125emX}}
 
\usepackage{multirow}    
\usepackage[normalem]{ulem} 

\newcommand{\va}{{\mathbf{a}}}
\newcommand{\vb}{{\mathbf{b}}}
\newcommand{\vc}{{\mathbf{c}}}
\newcommand{\vd}{{\mathbf{d}}}
\newcommand{\ve}{{\mathbf{e}}}

\newcommand{\vg}{{\mathbf{g}}}

\newcommand{\vp}{{\mathbf{p}}}

\newcommand{\vx}{{\mathbf{x}}}
\newcommand{\vy}{{\mathbf{y}}}
\newcommand{\vz}{{\mathbf{z}}}

\newcommand{\vA}{{\mathbf{A}}}
\newcommand{\vB}{{\mathbf{B}}}

\newcommand{\vD}{{\mathbf{D}}}
\newcommand{\vE}{{\mathbf{E}}}

\newcommand{\vI}{{\mathbf{I}}}

\newcommand{\vL}{{\mathbf{L}}}
\newcommand{\vM}{{\mathbf{M}}}

\newcommand{\vP}{{\mathbf{P}}}
\newcommand{\vQ}{{\mathbf{Q}}}
\newcommand{\vR}{{\mathbf{R}}}

\newcommand{\vX}{{\mathbf{X}}}
\newcommand{\vY}{{\mathbf{Y}}}
\newcommand{\vZ}{{\mathbf{Z}}}

\newcommand{\RR}{\mathbb{R}} 

\newcommand{\vzero}{{\mathbf{0}}}
\newcommand{\vareps}{{\varepsilon}}

\newtheorem{theorem}{Theorem}
\newtheorem{remark}{Remark}

\begin{document}

\title{Greedy coordinate descent method on non-negative quadratic programming
\thanks{This work in partly supported by the NSF grant DMS-1719549.}
}

\author{\IEEEauthorblockN{Chenyu Wu}
\IEEEauthorblockA{\textit{Department of Mathematical Sciences} \\
\textit{Rensselaer Polytechnic Institute}\\
Troy, NY 12180, USA \\
wuc10@rpi.edu}
\and
\IEEEauthorblockN{Yangyang Xu}
\IEEEauthorblockA{\textit{Department of Mathematical Sciences} \\
\textit{Rensselaer Polytechnic Institute}\\
Troy, NY 12180, USA \\
xuy21@rpi.edu}
}

\maketitle

\begin{abstract}
The coordinate descent (CD) method has recently become popular for solving very large-scale problems, partly due to its simple update, low memory requirement, and fast convergence. In this paper, we explore the greedy CD on solving non-negative quadratic programming (NQP). The greedy CD generally has much more expensive per-update complexity than its cyclic and randomized counterparts. However, on the NQP, these three CDs have almost the same per-update cost, while the greedy CD can have significantly faster overall convergence speed. We also apply the proposed greedy CD as a subroutine to solve linearly constrained NQP and the non-negative matrix factorization. Promising numerical results on both problems are observed on instances with synthetic data and also image data.
\end{abstract}

\begin{IEEEkeywords}
greedy coordinate descent, quadratic programming, nonnegative matrix factorization
\end{IEEEkeywords}

\section{Introduction}
The coordinate descent (CD) method is one of the most classic iterative methods for solving optimization problems. It dates back to 1950s \cite{hildreth1957quadratic} and is closely related to the Jacobi and Gauss-Seidel methods for solving a linear system. Compared to a full-update method, such as the gradient descent and the Newton's method, the coordinate update is simpler and cheaper, and also the CD method has lower memory requirement. Partly due to this reason, the CD method and its variants (such as coordinate gradient descent) have recently become particularly popular for solving very large-scale problems, under both convex and nonconvex settings (e.g., see \cite{tseng2009coordinate, nesterov2012efficiency, wright2015coordinate, shi2016primer, xu2013block, xu2015block, xu2017globally, hong2017iteration, dang2015stochastic, gao2019randomized, peng2016arock, razaviyayn2013unified}). Roughly speaking, the CD method, at each iteration, picks one (by a certain rule) out of possibly many coordinates and then updates it (in a certain way) to decrease the objective value. 

Early works (e.g., \cite{hildreth1957quadratic, luo1992convergence, tseng2001convergence}) on CD chose the coordinates cyclicly, or greedily such that the change to the variable or the decrease of the objective value is maximized \cite{chen2012maximum}. Since the pioneering work \cite{nesterov2012efficiency} that introduces random selection of the updated coordinate, many recent works focus on randomized CD methods (e.g., \cite{richtarik2014iteration, lin2014accelerated, liu2015asynchronous, peng2016arock, xu2019asynchronous}). Theoretically, the randomized CD can have faster convergence than the cyclic CD \cite{sun2019worst}. Computationally, the greedy CD is generally more expensive than the randomized and cyclic CD, namely, the latter two can be coordinate-friendly (CF) \cite{peng2016CF} while the greedy one may fail to. However, for some special-structured problems such as the $\ell_1$ minimization, the greedy CD is CF and can be faster than both the randomized and cyclic CD (e.g., \cite{li2009coordinate, peng2013parallel, nutini2015coordinate, nutini2018greed}) in theory and practice.

In this paper, we first explore the greedy CD to the non-negative quadratic programming (NQP). Similar to the cyclic and randomized CD methods, we show that the greedy CD is also CF for solving the NQP, by maintaining the full gradient of the objective and renewing it with $O(1)$ flops after each coordinate-update. Numerically, we demonstrate that the greedy CD can be significantly faster than the cyclic and randomized counterparts. We then apply the greedy CD as a subroutine in the framework of the augmented Lagrangian method (ALM) for the linearly constrained NQP and in the framework of the alternating minimization for the nonnegative matrix factorization (NMF) \cite{paatero1994positive, lee1999learning}. On both problems with synthetic and/or real-world data, we observe promising numerical performance of the proposed methods. 

\noindent\textbf{Notation.} 
The $i$-th component of a vector $\vx$ is denoted as $x_i$. $\vx_{<i}$ denotes  the subvector of $\vx$ of all components with indices less than $i$ and $\vx_{>i}$ with indices greater than $i$. Given a matrix $\vP$, we use $\vp_i$ for its $i$-th column and $\vp_{i:}$ for its $i$-th row. For a twice differentiable function $f$ on $\RR^n$, $\nabla_i f$ denotes the partial derivative about the $i$-th variable and $\nabla_i^2 f$ for the second-order partial derivative. 
$[n]$ represents the set $\{1,\ldots, n\}$.

\section{Greedy Coordinate Descent method}\label{sec:gcd-nqp}
We first briefly introduce the greedy CD method on a general coordinate-constrained optimization problem and then show the details on how to apply it to the NQP.

Let $f:\RR^n\to \RR$ be a differentiable function and for each $i\in [n]$, $X_i\subseteq \RR$ be a closed convex set. The greedy CD for solving $\min_{\vx\in\RR^n}\big\{ f(\vx): x_i\in X_i, \forall i\in[n]\big\}$ iteratively performs the update: $x_i^{k+1}=x_i^k$ if $i\neq i_k$, and
\begin{equation}\label{eq:gcd}
x_i^{k+1} = 
\argmin_{x_i\in X_i} f(\vx_{<i}^k, x_i, \vx_{>i}^k),  \text{ if } i = i_k,
\end{equation}
where $i_k$ is selected greedily by
\begin{equation}\label{eq:select-ik}
\textstyle i_k = \argmin_{i\in [n]} \left\{\min_{x_i\in X_i} f\big(\vx_{<i}^k, x_i, \vx_{>i}^k\big) \right\}.
\end{equation}
Notice that here we follow \cite{chen2012maximum} and greedily choose $i_k$ based on the objective value. In the literature, there are a few other ways to greedily choose $i_k$, based on the magnitude of partial derivatives or the change of coordinate update. We refer the readers to the review paper \cite{shi2016primer}.

In general, to choose $i_k$ by \eqref{eq:select-ik}, we need to solve $n$ one-dimensional minimization problem, and thus the per-update complexity of the greedy CD can be as high as $n$ times of that of a cyclic or randomized CD. Similar to \cite{li2009coordinate} that considers the LASSO, we show that the greedy CD on the NQP can have similar per-update cost as the cyclic and randomized CD. 

\subsection{Greedy CD on the NQP}
Now we derive the details on how to apply the aforementioned greedy CD on the following NQP:
\begin{equation}\label{eq:nqp}
\textstyle \min_{\mathbf{x}\geq \vzero} \;F(\mathbf{x})= \frac{1}{2}\mathbf{x}^\top \vP \mathbf{x}+ \vd^\top \mathbf{x}.
\end{equation}
Here, $\vP \in \mathbb R^{n\times n}$ is a given symmetric positive semidefinite (PSD) matrix, and 
$\vd\in \mathbb R^{n}$ is given. 
To have well-defined coordinate updates, we assume $P_{ii}>0,\,\forall\,i\in[n]$. 
The work \cite{thoppe2014greedy} has studied a greedy block CD on \emph{unconstrained} QPs, and it requires the matrix $\vP$ to be positive definite. Hence, our setting is more general, and more importantly, our algorithm can be used as a subroutine to solve a larger class of problems such as the linear equality-constrained NQP and the NMF, discussed in sections~\ref{sec:gcd-linqp} and \ref{sec:gcd-nmf}, respectively. We emphasize that our discussion on the greedy CD may be extended to other applications with separable non-smooth regularizers.

Suppose that the value of the $k$-th iterate is $\vx^k$. We define $G^{(k)}_i(x_i)=F(\vx_{<i}^k, x_i, \vx_{>i}^k)$.  
%
Since $F$ is a quadratic function, by the Taylor expansion, it holds
\begin{equation}
\textstyle G^{(k)}_i(x_i) = F(\mathbf{x}^k)+\nabla_i F(\mathbf{x}^k)  (x_i-x_i^k)
 + \frac{P_{ii}}{2}(x_i-x_i^k)^2.
 \label{eq:4}
\end{equation}
Let $\hat x_i^k$ be the minimizer of $G^{(k)}_i(x_i)$ over $x_i\ge0$. Then 
\begin{equation}\label{eq:xhat-ik}
\textstyle \hat x_i^k = \max\left(\textstyle 0,\ x_i^k - \frac{\nabla_i F(\mathbf{x}^k)}{P_{ii}}\right), \forall\, i\in [n],
\end{equation}
and thus the best coordinate by the rule in \eqref{eq:select-ik} is 
\begin{equation}\label{eq:ik-nqp}
i_k = \argmin_{i\in [n]}\left\{\textstyle \nabla_i F(\mathbf{x}^k)  (\hat x_i^k-x_i^k)
 + \frac{P_{ii}}{2}(\hat x_i^k-x_i^k)^2\right\}.
 \end{equation}
 
Notice that the most expensive part in \eqref{eq:xhat-ik} lies in computing $\nabla_i F(\mathbf{x}^k)$, which takes $O(n)$ flops. Hence, to obtain the best $i_k$ and thus a new iterate $\vx^{k+1}$, it costs $O(n^2)$. Therefore, performing $n$ coordinate updates will cost $O(n^3)$, which is order of magnitude larger than the per-update cost $O(n^2)$ by the gradient descent method. However, this naive implementation does not exploit the coordinate update, i.e., any two consecutive iterates differ at most at one coordinate. Using this fact, we show below that the greedy CD method on solving \eqref{eq:nqp} can be CF \cite{peng2016CF} by maintaining a full gradient, i.e., $n$ coordinate updates cost similarly as a gradient descent update.

Let $\vg^k=\nabla F(\vx^k)$. Provided that $\vg^k$ is stored in the memory, we only need $O(1)$ flops to have $\hat x_i^k$ defined in \eqref{eq:xhat-ik}, and thus to obtain the best $i_k$, it costs $O(n)$. Furthermore, notice that $\nabla F(\vx)=\vP\vx + \vd$. Hence,
$$\nabla F(\vx^{k+1}) = \nabla F(\vx^k) + \vP(\vx^{k+1}-\vx^k)=\vg^k + (\hat x_{i_k}^k - x_{i_k}^k) \vp_{i_k}.$$
Therefore, to renew $\vg$, it takes an additional $O(n)$ flops. This way, we need $O(n)$ flops to update the iterate and maintain full gradient from $(\vx^k, \vg^k)$ to $(\vx^{k+1}, \vg^{k+1})$, and thus completing $n$ greedy coordinate updates costs $O(n^2)$, which is similar to the cost of a gradient descent update.

\subsection{Stopping Condition}
A point $\vx^*\ge\vzero$ is an optimal solution to \eqref{eq:nqp} \emph{if and only if} $\vzero\in \nabla F(\vx^*) + \mathcal{N}_+(\vx^*)$, where $\mathcal{N}_+(\vx)=\big\{\vg\in\RR^n: g_i x_i \le 0,\, \forall\, i\in [n] \big\}$ denotes the normal cone of the non-negative orthant at $\vx\ge \vzero$. 
Therefore, $\vx^*$ should satisfy the following conditions for each $i\in[n]$:  $\nabla_i F(\vx^*) = 0$ if $x_i^*>0$, and $\nabla_i F(\vx^*) \ge 0$ if $x_i^*=0$. Let
$
I_0^k = \{i\in [n]: x_i^k = 0\},\ I_+^k = \{i\in [n]: x_i^k > 0\},
$
and 
\begin{equation}\label{eq:delta-k}
\textstyle \delta_k = \sqrt{\sum_{i\in I_0^k} \big[\min(0, \nabla_i F(\vx^k))\big]^2 + \sum_{i\in I_+^k} \big[\nabla_i F(\vx^k)\big]^2}.
\end{equation}
Then if $\delta_k$ is smaller than a pre-specified error tolerance, 
we can stop the algorithm.

\subsection{Pseudocode of the greedy CD}
Summarizing the above discussions, we have the pseudocode for solving the NQP \eqref{eq:nqp} in Algorithm~\ref{alg:gcd-nqp}.

\begin{algorithm}[h]\caption{Greedy CD for \eqref{eq:nqp}: $\mathrm{GCD}(\vP,\vd, \vareps,\vx^0)$}\label{alg:gcd-nqp}
\DontPrintSemicolon
{\small
\textbf{Input}: a PSD matrix $\vP\in\RR^{n\times n}$, $\vd\in\RR^n$, initial point $\vx^0\ge\vzero$, and an error tolerance $\vareps>0$.\;
\textbf{Overhead}: let  $k=0$, $\vg^0= \nabla F(\vx^0)$ and set $\delta_0$ by \eqref{eq:delta-k}.\; 
\While{$\delta_k>\vareps$}{
Compute $\hat x_i^k$ for all $i\in [n]$ by \eqref{eq:xhat-ik};\;
Find $i_k\in [n]$ by the rule in \eqref{eq:ik-nqp};\;
Let $x_i^{k+1}=x_i^k$ for $i\neq i_k$ and $x_i^{k+1}=\hat x_i^k$ for $i=i_k$;\;
Update $\vg^{k+1}=\vg^k + (x_{i_k}^{k+1}-x_{i_k}^k)\vp_{i_k}$;\;
Increase $k\gets k+1$ and compute $\delta_{k}$ by \eqref{eq:delta-k}.
}
\textbf{Output} $\vx^{k}$
}
\end{algorithm}
\vspace{-0.2cm}

\subsection{Convergence result}
The greedy CD that chooses coordinates based on the objective decrease has been analyzed in \cite{chen2012maximum}. We apply the results there to obtain the convergence of Algorithm~\ref{alg:gcd-nqp}. 

\begin{theorem}
Let $\{\vx^k\}$ be the sequence from Algorithm~\ref{alg:gcd-nqp}. Suppose that the lower level set $\mathcal{L}_0=\big\{\vx\ge \vzero: F(\vx)\le F(\vx^0)\big\}$ is compact. Then $F(\vx^k)\to F^*$ as $k\to \infty$, where $F^*$ is the optimal objective value of \eqref{eq:nqp}. 
\end{theorem}

\begin{proof}
Since $F(\vx^k)$ is decreasing with respect to $k$ and $\mathcal{L}_0$ is compact, the sequence $\{\vx^k\}$ has a finite cluster point $\bar\vx$, and $F(\vx^k)\to F(\bar\vx)$ as $k\to \infty$. Now notice that the minimizer of $G^{(k)}_i$ in \eqref{eq:4} is unique for each $i$ since $P_{ii}>0$. Hence, it follows from \cite[Theorem 3.1]{chen2012maximum} that $\bar\vx$ must be a stationary point of \eqref{eq:nqp} and thus an optimal solution because $\vP$ is PSD. Therefore, $F(\bar\vx)=F^*$, and this completes the proof.
\end{proof}

\begin{remark}
It is not difficult to show that $\textstyle F(\vx^{k+1})-F(\vx^k)\le -\frac{P_{ii}}{2}\|\vx^{k+1}-\vx^k\|^2.$
In addition, by \cite[Theorem 18]{wang2014iteration}, the quadratic function $F$ satisfies the so-called global error bound. Hence, it is possible to show a globally linear convergence result of Algorithm~\ref{alg:gcd-nqp}. Due to the page limitation, we do not extend the detailed discussion here, but instead we will explore it in an extended version of the paper.
\end{remark}

\subsection{Comparison to the cyclic and randomized CD methods}
We compare the greedy CD to its cyclic and randomized counterparts and also the accelerated projected gradient method FISTA \cite{beck2009fast}. Two random NQP instances were generated. For the first one, we set $n=5,000$ and generated a symmetric PSD matrix $\vP$ and the vector $\vd$ by the normal distribution. For the second one, we set $n=1,000$, $\vP=0.1\vI+0.9 \vE$ and $\vd = -10 \ve$, where $\vI$ denotes the identity matrix, and $\vE$ and $\ve$ are all-ones matrix and vector. The $\vP$ in the second instance was used to construct a difficult unconstrained QP for the cyclic CD in \cite{sun2019worst,lee2018random}. Figure~\ref{fig:nqp} plots the objective error produced by the three different CD methods and FISTA. From the plots, we clearly see that the greedy CD performs significantly better than the other two CDs and FISTA on both instances. 

\vspace{-0.1cm}
\begin{figure}[h]
\begin{center}
\includegraphics[width=0.2\textwidth]{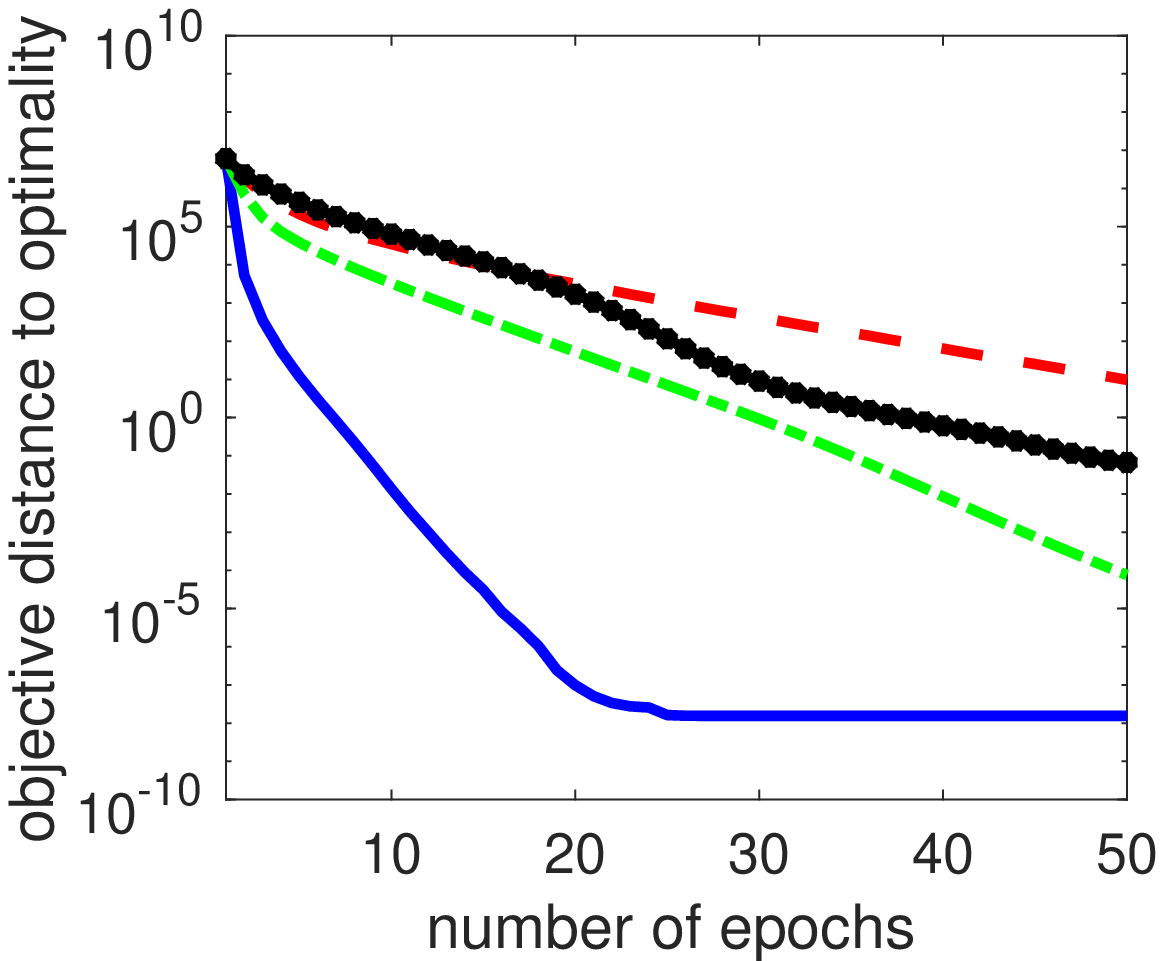}
\includegraphics[width=0.2\textwidth]{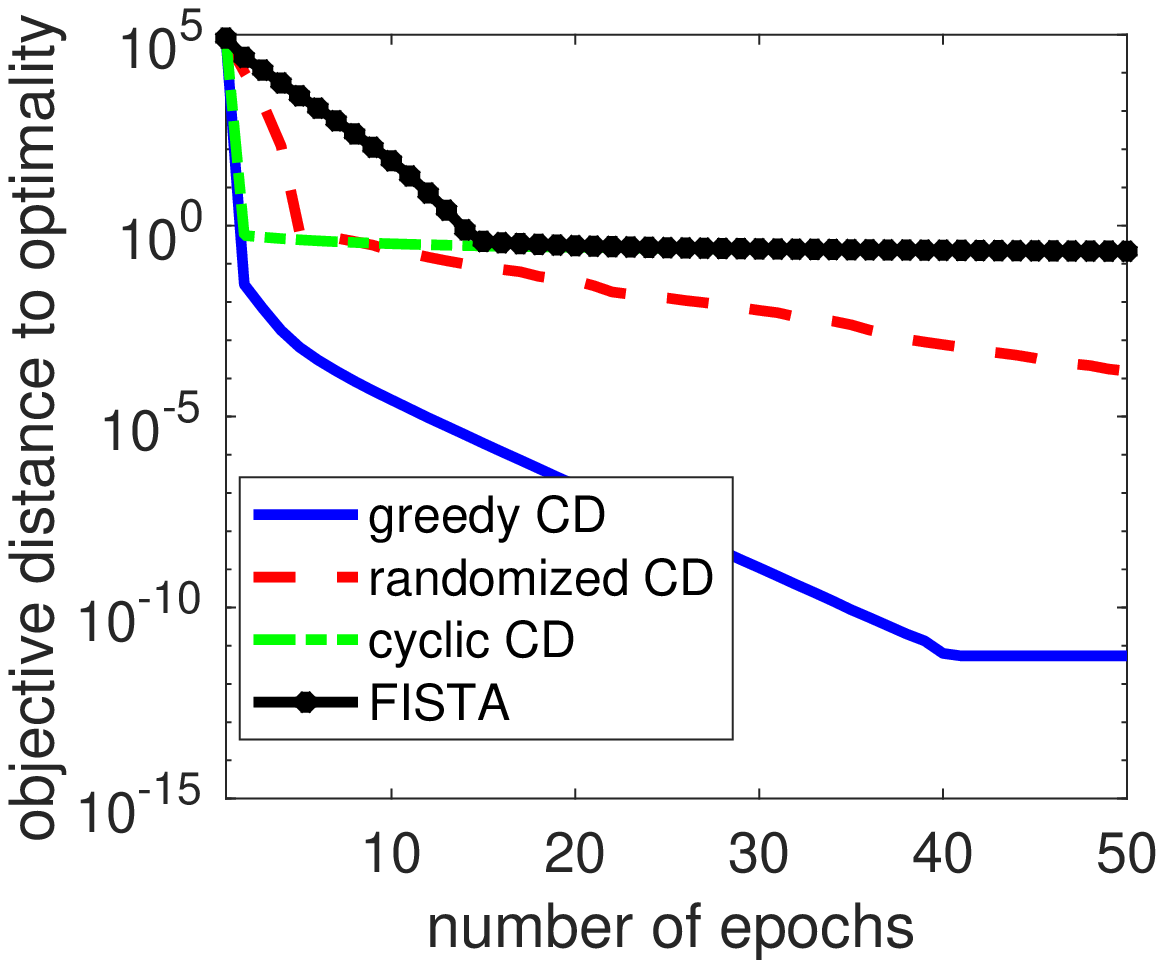}
\end{center}
\vspace{-0.3cm}
\caption{Comparison of three different CDs and FISTA on two NQP instances. Left: dimension $n=5,000$ and $\vP$ generated by normal distribution; Right: $n=1,000$ and $\vP=0.1\vI+0.9 \vE$.}\label{fig:nqp}
\vspace{-0.5cm}
\end{figure}

\section{Linear equality-constrained non-negative quadratic programming}\label{sec:gcd-linqp}
In this section, we apply Algorithm~\ref{alg:gcd-nqp} as a subroutine in the inexact ALM framework to solve the NQP with linear constraints. More specifically, we consider the problem
\vspace{-0.1cm}
\begin{equation}\label{eq:lin-nqp}
\textstyle \min_{\mathbf{x}\geq \vzero} F(\mathbf{x}) = \frac{1}{2}\mathbf{x}^\top \vQ \mathbf{x}+\vc^\top \mathbf{x},\ \mathrm{s.t.}\ \vA\mathbf{x}=\vb, 
\vspace{-0.1cm}
\end{equation}
where $\vQ\in\RR^{n\times n}$ is a PSD matrix, and $\vA\in\RR^{m\times n}$. 
The ALM \cite{powell1969method, rockafellar1973multiplier} is perhaps the most popular method for solving nonlinear functional constrained problems. Applied to \eqref{eq:lin-nqp}, it iteratively performs the updates:
\begin{subequations}\label{eq:ialm}
\begin{align}
&\textstyle \vx^{k+1} = \argmin_{\vx\geq \vzero} L_{\beta_k} (\vx,\vy^k)\label{eq:ialm-x}\\[0.1cm]
&\textstyle \vy^{k+1} = \vy^k+\beta_k(\vA\vx^{k+1}-\vb)\label{eq:ialm-y}.
\vspace{-0.2cm}
\end{align}
\end{subequations}
Here, $\vy\in\RR^m$ is the Lagrange multiplier,  and
\begin{equation}
\textstyle L_{\beta}(\vx,\vy) = F(\vx)+\vy^\top(\vA\vx-\vb)+\frac{\beta}{2}\|\vA\vx-\vb\|^2
\vspace{-0.1cm}
\end{equation}
is the AL function with a penalty parameter $\beta>0$.

\subsection{inexact ALM with greedy CD}
In the ALM update, the $\vy$-update is easy. However, the $\vx$-subproblem \eqref{eq:ialm-x} in general requires an iterative solver. 
On solving \eqref{eq:lin-nqp}, we can rewrite the AL function as
$$\textstyle L_{\beta} (\vx,\vy)=\frac{1}{2}\vx^\top (\vQ+\beta \vA^\top\vA) \vx+(\vc+\vA^\top\vy-\beta \vA^\top \vb )^\top\vx,$$
which is a quadratic function. Hence, \eqref{eq:ialm-x} is an NQP, and we propose to apply the greedy CD derived previously to solve it. 
The pseudocode of the proposed method is shown in Algorithm~\ref{alg:ialm-nqp}. 

\begin{algorithm}\caption{inexact ALM with greedy CD for \eqref{eq:lin-nqp}}\label{alg:ialm-nqp}
\DontPrintSemicolon
{\small
\textbf{Input}: a PSD matrix $\vQ\in\RR^{n\times n}$, $\vc\in\RR^n$, $\vA\in\RR^{m\times n}$, $\vb\in\RR^m$, and $\vareps>0$.\;
\textbf{Initialization}: $\vx^0\in\RR^n, \vy^0=\vzero$, and $\beta_0  > 0$; set $k=0$.\;
\While{a stopping condition not satisfied}{
Let $\vP= \vQ+\beta_k \vA^\top\vA$ and $\vd = \vc+\vA^\top\vy-\beta_k \vA^\top \vb$.\;
Choose $\vareps_k\le \vareps$ and let $\vx^{k+1}=\mathrm{GCD}(\vP,\vd,\vareps_k,\vx^k)$.\;
Obtain $\vy^{k+1}$ by \eqref{eq:ialm-y}.\;
Choose $\beta_{k+1}\ge\beta_k$, and increase $k\gets k+1$.
}
}
\end{algorithm}
\vspace{-0.2cm}

Notice that each $\vx^{k+1}$ satisfies $\mathrm{dist}\big(\vzero, \nabla_\vx L_{\beta_k}(\vx^{k+1},\vy^k)+\mathcal{N}_+(\vx^{k+1})\big)\le\vareps_k$. Hence, by the update of $\vy^{k+1}$, it holds 
$$\mathrm{dist}\left(\vzero, \nabla F(\vx^{k+1}) + \vA^\top\vy^{k+1} +\mathcal{N}_+(\vx^{k+1})\right)\le\vareps_k,$$
namely, the dual residual is always no larger than $\vareps_k$. Since $\vareps_k\le \vareps$, the output $(\vx^{k+1},\vy^{k+1})$ violates the KKT conditions at most $\vareps$ in terms of both primal and dual feasibility, if we stop the algorithm once $\|\vA\vx^{k+1}-\vb\|\le \vareps$.

\subsection{Convergence result}
We can apply the results in \cite{rockafellar1973multiplier, xu2019ialm, li-xu2020ialm} to obtain the convergence of Algorithm~\ref{alg:ialm-nqp} based on the actual iterate.

\begin{theorem}
Suppose that for each $i\in[n]$, $Q_{ii}>0$ or $\va_i\neq \vzero$. Let $\{\vx^k\}_{k\ge 0}$ be the sequence from Algorithm~\ref{alg:ialm-nqp}. If $\vareps_k\to 0$ and $\beta_k\to\infty$, then $|F(\vx^k)-F^*| \to 0$ and $\|\vA\vx^k-\vb\|\to0$,
where $F^*$ denotes the optimal objective value of \eqref{eq:lin-nqp}.
\end{theorem}

\section{Non-negative matrix factorization}\label{sec:gcd-nmf}
In this section, we consider the non-negative matrix factorization (NMF). It aims to factorize a given non-negative matrix $\vM\in\RR^{m\times n}_+$ into two low-rank non-negative matrices $\vX$ and $\vY$ such that $\vM\approx \vX\vY^\top$. The factor matrix $\vX$ plays a role of basis while $\vY$ contains the coefficient. Due to the non-negativity, NMF can be used to learn local features of an objective \cite{lee1999learning} and has better interpretability than the principal component analysis (PCA). Measuring the approximation error by the Frobenius norm, one can model the NMF as
\begin{equation}\label{eq:2}
\textstyle \min_{\vX, \vY} \frac{1}{2}\|\vX\vY^\top-\vM\|_F^2,\ \mathrm{s.t.}\ \vX\in\RR^{m\times r}_+, \, \vY\in\RR^{n\times r}_+, 
\end{equation}
where $\vM\in\RR^{m\times n}_+$ is given, and $r$ is a user-specified rank.

\subsection{Alternating minimization with greedy CD}
The objective of \eqref{eq:2} is non-convex jointly with respect to $\vX$ and $\vY$. However, it is convex with respect to one of $\vX$ and $\vY$ while the other is fixed. For this reason, one natural way to solve \eqref{eq:2} is the alternating minimization (AltMin), which iteratively performs the update 
\begin{subequations}\label{eq:altm-nmf}
\begin{align}
&\textstyle\vX^{k+1}=\argmin_{\vX\geq \vzero} \frac{1}{2}\|\vX(\vY^k)^\top-\vM\|_F^2 \label{eq:altm-nmf-x},\\[0.1cm]
&\textstyle \vY^{k+1}=\argmin_{\vY\geq \vzero} \frac{1}{2}\|(\vX^{k+1})\vY^\top-\vM\|_F^2.\label{eq:altm-nmf-y}
\end{align}
\end{subequations}
Both subproblems are in the form of
$\min_{\vZ\in\RR^{r\times p}} \frac{1}{2}\|\vA\vZ-\vB\|_F^2,$
which is equivalent to solving $p$ independent NQPs
$$\textstyle \min_{\vz_i\in\RR^{r}} \frac{1}{2}\|\vA\vz_i-\vb_i\|^2, \, i\in [p].$$
Hence, we can apply the greedy CD derived in section \ref{sec:gcd-nqp} to solve the $\vX$-subproblem and $\vY$-subproblem in \eqref{eq:altm-nmf}, by breaking them respectively into $m$ and $n$ independent NQPs. The pseudocode is shown in Algorithm~\ref{alg:altm-nmf}. In Lines~4 and 8, we rescale those two factor matrices such that they have balanced norms. This way, neither of them will blow up or diminish, and the resulting NQP subproblems are relatively well-conditioned. Numerically, we observe that the rescaling technique can significantly speed up the convergence.

Notice that it is straightforward to extend our method to the non-negative tensor decomposition \cite{shashua2005non}. Due to the page limitation, we do not give the details here but leave it to an extended version of this paper.

\begin{algorithm}\caption{AltMin with greedy CD for \eqref{eq:2}}\label{alg:altm-nmf}
\DontPrintSemicolon
{\small
\textbf{Input}: $\vM \in \mathbb R^{m\times n}_+$, and rank $r$.\;
\textbf{Initialization}: $\vX^0 \in \mathbb R^{m\times r}_+$ and $\vY^0\in \mathbb R^{n\times r}_+$.\;
\For {$k=0,1,\ldots$}{
Rescale $\vX^k$ and $\vY^k$ to $\|\vx_i^k\|=\|\vy_i^k\|,\forall\, i\in[r]$.\;
Let $\vP = (\vY^k)^\top \vY^k$ and $\vD =  -(\vY^k)^\top \vM^\top$.\;
Choose $\vareps_k>0$.\;
Compute $\vx_{i:}^{k+1} = \mathrm{GCD}(\vP, \vd_i,\vareps_k, \vx^k_{i:})$, for $i\in[m]$.\;
Rescale $\vX^{k+1}$ and $\vY^k$ to $\|\vx_i^{k+1}\|=\|\vy_i^k\|,\forall\, i\in[r]$.\;
Let $\vP = (\vX^{k+1})^\top \vX^{k+1}$ and $\vD =  -(\vX^{k+1})^\top \vM$.\;
Compute $\vy_{i:}^{k+1} = \mathrm{GCD}(\vP, \vd_i,\vareps_k, \vy^k_{i:})$, for $i\in[n]$.
}
}
\end{algorithm}

\vspace{-0.4cm}

\subsection{Convergence result}
The convergence of the AltMin has been well-studied; see \cite{grippo2000convergence} for example. Although we rescale the two factor matrices and inexactly solve each subproblem, it is not difficult to adapt the existing analysis and obtain the convergence of Algorithm~\ref{alg:altm-nmf} as follows.
\begin{theorem}\label{thm:altm-nmf}
Let $\{(\vX^k,\vY^k)\}_{k\ge 0}$ be the sequence generated from Algorithm~\ref{alg:altm-nmf} with $\vareps_k\to 0$. Then any finite limit point of the sequence is a stationary point of \eqref{eq:2}.
\end{theorem}

\section{Numerical experiments}\label{sec:numerical}
In this section, we test Algorithm~\ref{alg:ialm-nqp} on Gaussian random-generated instances of \eqref{eq:lin-nqp} and compare it to the MATLAB built-in solver \verb|quadprog|. Also, we test Algorithm~\ref{alg:altm-nmf} on three instances of the NMF \eqref{eq:2}, two with synthetic data and another with face image data. For the tests on \eqref{eq:lin-nqp}, we generated three different-sized instances. In Algorithm~\ref{alg:ialm-nqp}, we set $\vareps_k=10^{-3},\,\forall\, k$, for the subroutine GCD, and the algorithm was stopped once $\|\vA\vx^k-\vb\|\le \vareps$ with $\vareps=10^{-2}$ or $10^{-3}$.  For the tests on \eqref{eq:2} with synthetic data, we obtain $\vM=\vL\vR^\top$, where $\vL\in\RR^{m\times r}_+$ and $\vR\in\RR^{n\times r}_+$ were respectively generated by MATLAB's code \verb|max(0,randn(m,r))| and \verb|max(0,randn(n,r))|. One dataset was generated with $m=n=1,000, r= 50$ and the other with $m=n=5,000, r= 100$. For the other test on \eqref{eq:2}, we used the CBCL face image data \cite{hoyer2004non}, which consists of 6,977 images of size $19\times 19$. We picked the first 2,000 images and vectorized each image into a vector. This way, we formed a non-negative $\vM\in\RR^{361\times 2000}$, and we set $r=30$. In Algorithm~\ref{alg:altm-nmf}, we set $\vareps_k=10^{-3},\,\forall\, k$, for the subroutine GCD on the test with synthetic data and $\vareps_k=0.1,\,\forall\, k$, on the face image data. Different tolerances were adopted here because the image data does not admit an exact factorization.

\begin{table}\caption{Results by Algorithm~\ref{alg:ialm-nqp} on three different-sized instances of \eqref{eq:lin-nqp} and the speed comparison with MATLAB function $\mathrm{quadprog}$. Here, tol. is the $\vareps$ in Algorithm~\ref{alg:ialm-nqp}; obj. relerr is computed by $\frac{|F(\vx)-F^*|}{|F^*|}$; res. relerr is by $\frac{\|\vA\vx-\vb\|}{\|\vb\|}$; time1 is the running time (sec.) by the proposed method; time2 is by $\mathrm{quadprog}$.}\label{table:lin-nqp}
\vspace{-0.3cm}
\begin{center}
\begin{tabular}{|c||cccc||c|}
\hline
problem size & tol. & obj. relerr & res. relerr &  time1& time2\\\hline\hline
$m = 200$ & $0.01$ & 2.758e-05 & 5.192e-04 &  0.42 &  \multirow{2}{*}{0.32}\\\cline{2-5}
 $n = 1000$ & $0.001$ & 1.118e-06 & 2.811e-05 &  0.76 &  \\\hline\hline
$m = 1000$ & $0.01$ & 3.714e-06 & 1.138e-04 & 16.68 & \multirow{2}{*}{24.92}\\\cline{2-5}
 $n = 5000$ & $0.001$ & 2.257e-07 & 8.074e-06 & 30.85 & \\\hline\hline
 $m = 2000$ & $0.01$ & 4.361e-07 & 2.021e-05 & 71.90 & \multirow{2}{*}{217.41}\\\cline{2-5}
$n = 10000$ & $0.001$ & 3.795e-07 & 1.419e-05 & 133.07 & \\\hline
\end{tabular}
\end{center}
\vspace{-0.5cm}
\end{table}

\begin{figure}[h]
\begin{center}
\includegraphics[width=0.15\textwidth]{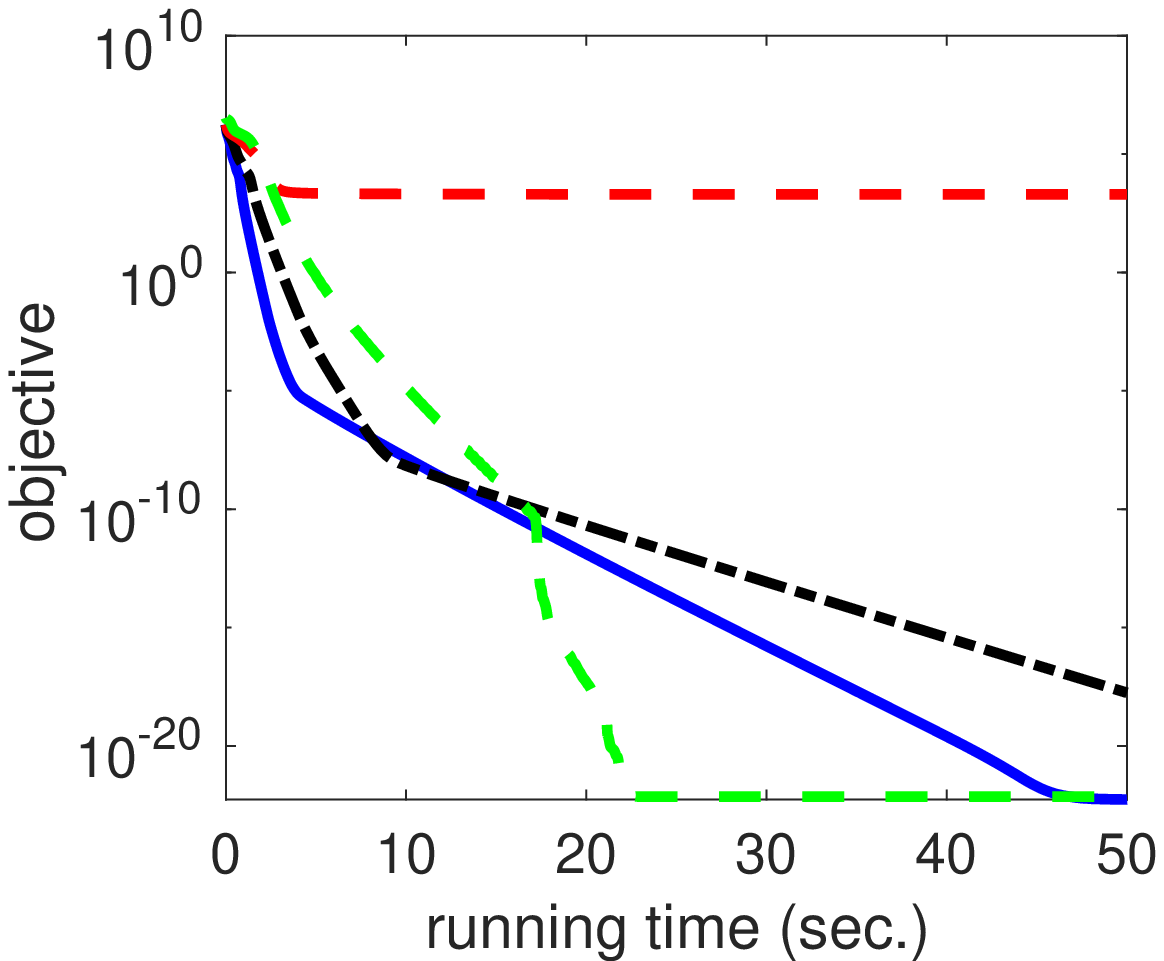}
\includegraphics[width=0.15\textwidth]{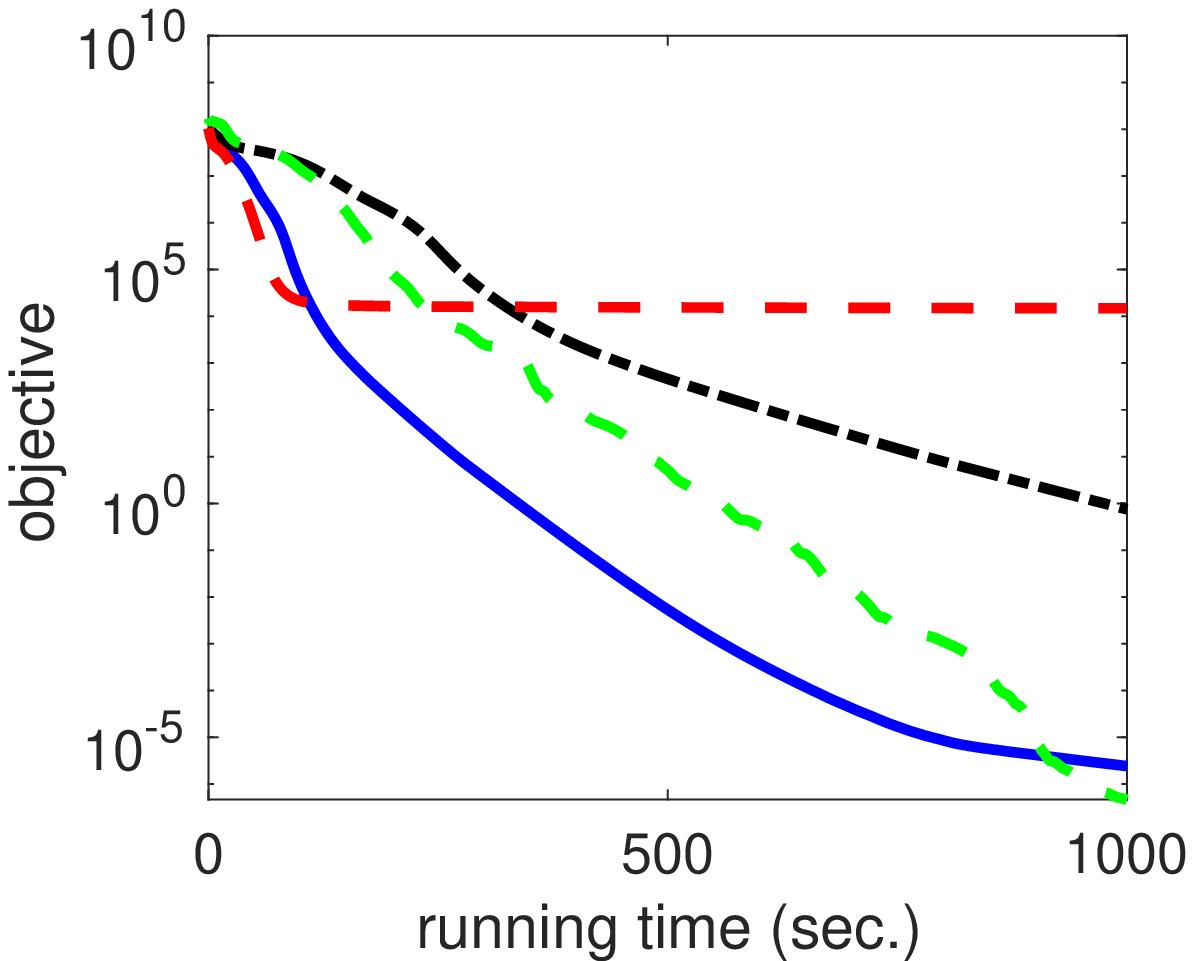}
\includegraphics[width=0.15\textwidth]{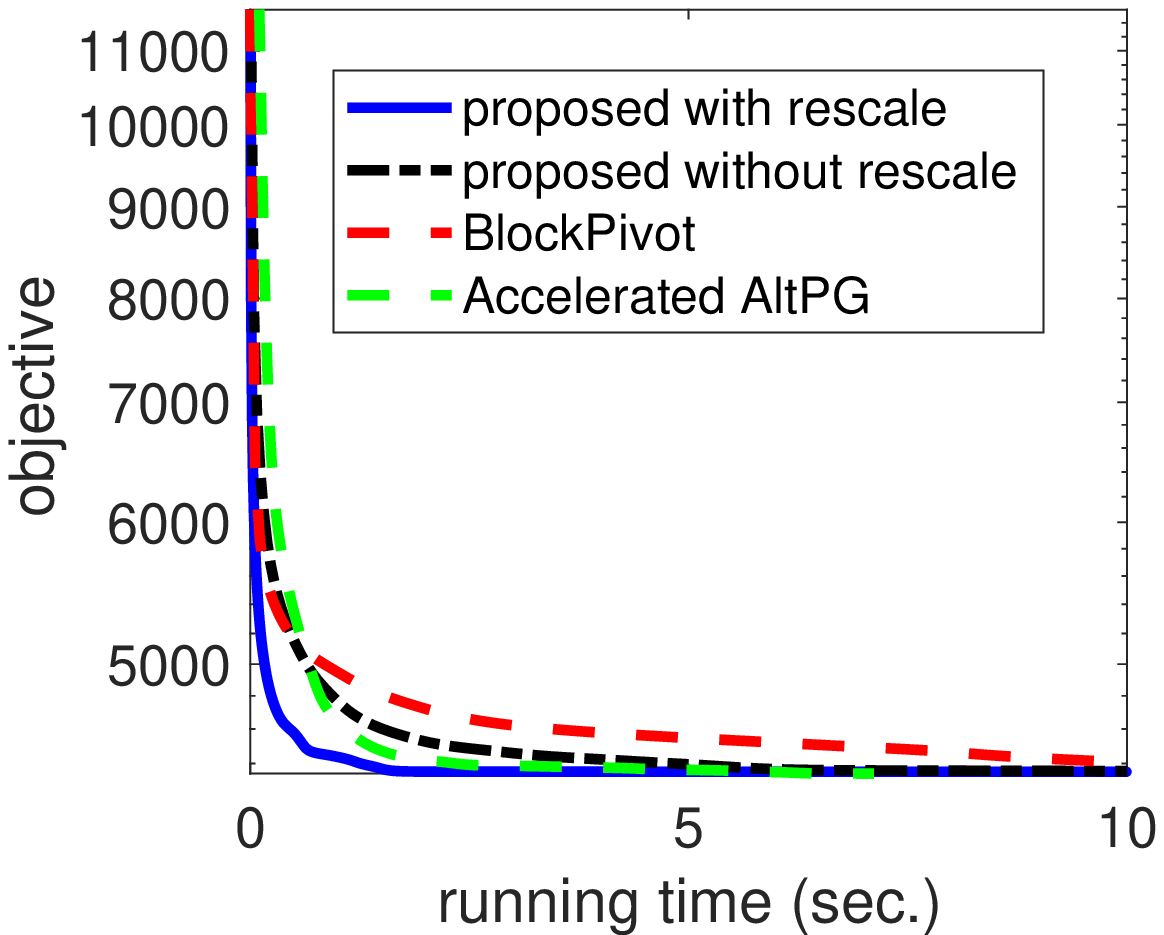}
\end{center}
\vspace{-0.2cm}
\caption{Comparison of the proposed method (Algorithm~\ref{alg:ialm-nqp}) with and without the rescaling to the BlockPivot method in \cite{kim2008toward} and the accelerated AltPG in \cite{xu2013block}. Left: synthetic data with $m=n=1,000, r=50$; Middle: synthetic data with $m=n=5,000,r=100$; Right: CBCL face image data set}\label{fig:nmf}
\vspace{-0.2cm}
\end{figure}

The results for the tests on \eqref{eq:lin-nqp} are shown in Table~\ref{table:lin-nqp}. From the results, we see that our method can yield medium-accurate solutions. For the middle-sized instance, our method can be faster than MATLAB's solver when $\vareps=10^{-2}$, and for the large-sized instance, our method is faster under both settings of $\vareps=10^{-2}$ and $10^{-3}$. This implies that for solving large-scale linear-constrained NQP, the proposed method can outperform MATLAB's solver if a medium-accurate solution is required. The results for the tests on \eqref{eq:2} are plotted in Figure~\ref{fig:nmf}, where we compared Algorithm~\ref{alg:altm-nmf} with the BlockPivot method in \cite{kim2008toward} and the accelerated AltPG method in \cite{xu2013block}. The BlockPivot is also an AltMin, but different from our greedy CD, it uses an active-set-type method as a subroutine to exactly solve each subproblem. The accelerated AltPG performs block proximal gradient update to $X$ and $Y$ alternatingly, and it uses extrapolation technique for acceleration. From the results, we see that the proposed method performs significantly better than BlockPivot, which seems to be trapped at a local solution for the two synthetic cases. Compared to the accelerated AltPG, the proposed method can be faster in the beginning to obtain a medium accuracy. In addition, the proposed method converges faster with the rescaling than that without the rescaling on the instances with large-sized synthetic data and the face image data. That could be because the subproblems may be bad-conditioned if rescaling is not applied.

\newpage

\bibliographystyle{abbrv}
\bibliography{ref}

\end{document}